\documentclass{amsart}
\usepackage{amsmath}
\usepackage{amsfonts}
\usepackage{amssymb}
\usepackage{graphicx}
\usepackage{float, verbatim}
\usepackage{enumerate}
\usepackage{color}
\usepackage{subcaption}
\usepackage{tikz}

\newtheorem{thm}{Theorem}[section]

\newtheorem{prop}[thm]{Proposition}

\begin{document}

\title{Root sets of polynomials and power series with finite choices of coefficients}

\author{Simon Baker and Han Yu}
\address{Simon Baker\\ Mathematics institute\\ University of Warwick\\ Coventry\\ CV4 7AL\\ UK}
\email{simonbaker412@gmail.com}
\address{Han Yu\\
School of Mathematics \& Statistics\\University of St Andrews\\ St Andrews\\ KY16 9SS\\ UK \\ }
\curraddr{}
\email{hy25@st-andrews.ac.uk}
\thanks{}



\date{}

\dedicatory{}

\begin{abstract}
Given $H\subseteq \mathbb{C}$ two natural objects to study are the set of zeros of polynomials with coefficients in $H$, $$\{z\in \mathbb{C}: \exists k>0,\, \exists (a_n)\in H^{k+1}, \sum_{n=0}^{k}a_{n}z^n=0\},$$ and the set of zeros of power series with coefficients in $H$, $$\{z\in\mathbb{C}: \exists (a_n)\in H^{\mathbb{N}}, \sum_{n=0}^{\infty} a_nz^n=0\}.$$ In this paper we consider the case where each element of $H$ has modulus $1$. The main result of this paper states that for any $r\in(1/2,1),$ if $H$ is  $2\cos^{-1}(\frac{5-4|r|^2}{4})$-dense in $S^1,$ then the set of zeros of polynomials with coefficients in $H$ is dense in $\{z\in \mathbb{C}: |z|\in [r,r^{-1}]\},$ and the set of zeros of power series with coefficients in $H$ contains the annulus $\{z\in \mathbb{C}: |z|\in[r,1)\}$. These two statements demonstrate quantitatively how the set of polynomial zeros/power series zeros fill out the natural annulus containing them as $H$ becomes progessively more dense.

\end{abstract}

\maketitle

\section{Introduction}

Let $H\subseteq \mathbb{C}$ be a finite set. Given such a $H$ we define the root set of polynomials with coefficients in $H$ to be: $$R(H):=\{z\in \mathbb{C}: \exists k>0,\, \exists (a_n)\in H^{k+1}, \sum_{n=0}^{k}a_{n}z^n=0\}.$$ Similarly, we define the root set of power series with coefficients in $H$ to be
 $$R^*(H):=\{z\in \mathbb{C}: \exists (a_n)\in H^{\mathbb{N}}, \sum_{n=0}^{\infty}a_{n}z^n=0\}.$$ The study of the sets $R(H)$ and $R^{*}(H)$ can be dated back to Littlewood \cite{Lit} who studied the case where $H=\{-1,1\}.$ Since then many related works have appeared, most notable amongst these are the number theoretic results of Beaucoup, Borwein, Boyd and Pinner \cite{BBBP}, and Borwein, Erd\'{e}lyi and Littmann \cite{BEL}, who studied the distribution of roots and multiple roots. Related work also appeared in Bousch \cite{Bo}, where it was shown that $R(\{-1,1\})$ is dense in $\{z:|z|^4 \in [1/2,2]\}.$ In Shmerkin and Solomyak \cite{SS} some measure theoretic and topological properties of $R(\{-1,0,1\})$ are studied in detail.

In what follows we will adopt the following notational conventions:
$$S^{r}:=\{z\in \mathbb{C}:|z|=r\},\, B(z,r):=\{z'\in \mathbb{C}:|z'-z|<r\}, $$ and given some interval $I$ in $\mathbb{R}$ let $$A_{I}:=\{z\in \mathbb{C}:|z|\in I\}.$$ In this paper we focus on the case where $H$ is a subset of the unit circle $S^1$. Under this assumption it is straightforward to show that $$R(H)\subseteq A_{[1/2,2]}\textrm{ and }R^{*}(H)\subseteq A_{[1/2,1)}.$$ Intuitively, one might expect that if we allowed $H$ to become a progressively more dense subset of $S^1,$ then $R(H)$ and $R^{*}(H)$ would begin to fill out their respective annuli. The main result of this paper shows that this intuition is correct.

Before stating this result we need to define a metric on $S^1$ to properly quantify the density of $H$. Given $e^{i\theta}, e^{i\theta'}\in S^1$ let $d(e^{i\theta},e^{i\theta'})= \min \{|\theta-\theta'|, |2\pi -(\theta-\theta')|\}.$ This metric outputs the interior angle of the sector of $S^1$ determined by the two radii $e^{i\theta}$ and $e^{i\theta'}.$
\begin{thm}
\label{main theorem}
Fix $r\in (1/2,1)$. Suppose $H\subseteq S^1$ is $2\cos^{-1}(\frac{5-4r^2}{4})$-dense. Then $A_{[r,1)}\subseteq R^{*}(H)$ and $R(H)$ is dense in $A_{[r,r^{-1}]}.$
\end{thm}

The sets $R(H)$ and $R^{*}(H)$ are related by the following formula.
\begin{prop}
\label{root relations}
Let $H\subseteq \mathbb{C}$ be any finite set, then the following relations hold:
$$R(H)=\frac{1}{R(H)}$$ and
$$\overline{R(H)}\cap B(0,1)=R^{*}(H)\cap B(0,1).$$
\end{prop}In the statement of Proposition \ref{root relations}, $\overline{A}$ denotes the closure of a set $A,$ and $\frac{1}{A}$ denotes the set $\{z\in \mathbb{C}:z^{-1}\in A\}.$

\begin{proof}
Given $z\in \mathbb{C}$ suppose there is  polynomial $P\in H[x]$ such that  
\[
P(z)=\sum_{n=0}^{k}a_n z^n=0.
\]
We can construct another polynomial $Q\in H[x]$ such that $Q(1/z)=0$. Just consider $Q(x)=x^kP(\frac{1}{x})$ with $k=\deg P$. Therefore, whenever $z\in R(H)$ we also have $1/z\in R(H)$. This proves our first relation. 

Now we shall show that
\[
\overline{R(H)}\cap B(0,1)=R^{*}(H)\cap B(0,1).
\]

Without loss of generality we can assume that $H\subseteq\{z: |z|\leq 1\}$.
If $z^*\in \overline{R(H)}\cap B(0,1)$ then we can find a sequence $(z_i)\in R(H)^{\mathbb{N}}$ with:
\[
z_i\to z^*.
\]
Moreover, since $z^*\in B(0,1)$ there exists a positive number $M$ such that $1<M<\frac{1}{|z^*|}$. Now let us consider any polynomial in $H[x]$
\[
P(x)=\sum_{n=0}^{k}a_n x^n.
\]
The following result holds for $|z_i|\leq M|z^*|$
\begin{eqnarray*}
	|P(z^*)-P(z_i)|&\leq&\sum_{n=0}^{k}|a_n||(z^*)^n-z_i^n|=\sum_{n=0}^{k}|a_n||z^*-z_i||(z^*)^{n-1}+(z^*)^{n-2}z_i+...+z_i^{n-1}|\\
	&\leq& \sum_{n=0}^{k}|a_n||z^*-z_i|n(M|z^*|)^{n-1}\\
	&\leq& |z^*-z_i|\sum_{n=0}^{k}|a_n|n(M|z^*|)^{n-1}.
\end{eqnarray*}

Since $|a_n|\leq 1$ and $M|z^*|<1$ the latter summation can be bounded uniformly with respect to $k$, namely:
\[
\sum_{n=0}^{k}|a_n|n(M|z^*|)^{n-1}\leq C.
\]
Where $C>0$ is a constant that only depends upon $z^*$. 

Each $z_i$ is the root of some polynomial $P_i\in H[x]$, in which case by the above, for $i$ sufficiently large we have
\begin{equation}
\label{power series bound}
|P_i(z^*)|=|P_i(z^*)-P(z_i)|\leq C|z^*-z_i|.
\end{equation}

For the sequence $(P_i)$ there is either a uniform upper bound for the degrees of the $P_i$, or there exists a subsequence along which the degrees tend to infinity. In the first case there must exist a polynomial $Q\in H[x]$ and a subsequence $(P_{i_{j}})$ such that $P_{i_{j}}=Q$ for all $i_j.$ By \eqref{power series bound} we must then have $Q(z^*)=0$. Suppose $\deg Q =L,$ then 
$$T(x)=Q(x)\sum_{n=0}^{\infty}x^{n(L+1)}$$ is a power series with digits in $H$. For this particular power series we clearly have $T(z^*)=0$. Therefore in the first case we have $z^*\in R^*(H)$. Now suppose there exists a subsequence $(P_{i_j})$ such that $\deg P_{i_j}\to\infty.$ Via a diagonalisation argument, one can assume without loss of generality that there exists a sequence $(a_n)\in H^{\mathbb{N}}$ and an increasing sequence of natural numbers $(l_n),$ such that for all $i_j\geq l_n$ the coefficient of the degree $n$ term of $P_{i_j}$ is $a_n.$ In other words, as the $i_j$ become sufficiently large the lower order terms of the $P_{i_j}$'s start to coincide. It follows from \eqref{power series bound} then that for this sequence $(a_n)$ we must have
$$\sum_{n=0}^{\infty}a_n(z^*)^n=0.$$ Therefore $z^*\in R^*(H)$ and $\overline{R(H)}\cap B(0,1)\subseteq R^{*}(H)\cap B(0,1).$

Now suppose $z^*\in R^*(H)\cap B(0,1).$ Then there is a sequence $(a_n)\in H^{\mathbb{N}}$ such that
\[
\sum_{n=0}^{\infty}a_n (z^*)^n=0.
\]
This series is absolutely and uniformly convergent in $B(0,c)$ for any $0<c<1$. Since $z^*\in B(0,1)$ it is contained in one of these sets for $c$ sufficiently close to $1$. We see that the function
\[
P(x)=\sum_{n=0}^{\infty}a_n x^n
\]
is holomorphic on the interior of the unit disc and therefore the roots of $P$ must form a discrete set. Since $z^*$ belongs to the root set of $P$ there must exist $r>0$ such that
\[
\{z\in\mathbb{C}:P(z)=0\}\cap B(z^*,r)=\{z^*\}.
\]
Now suppose that $z^*\notin\overline{R(H)}$, then there exists a ball $B(z^*,r')$ such that $B(z^*,r')\subseteq B(z^*,r)$ and
\[
R(H)\cap \overline{B(z^*,r')}=\emptyset.
\]
We can then consider the following integral with $P_N(x)=\sum_{n=0}^{N}a_n x^n$
\[
I_N=\int_{\partial B(z^*,r')}\frac{P_N'(x)}{P_N(x)}dx.
\]
By our conditions on $r'$ we see that $P_N$ has no zeros in $\overline{B(z^*,r')}$ for all $N\in \mathbb{N}$. Therefore by the argument principle (see \cite[p.~152]{Ahl}) we must have $I_N\equiv 0.$  One can also assume that $r'$ is sufficiently small that $P_N$ converges to $P$ absolutely and uniformly. Therefore
\[
0=\lim_{N\to\infty}I_N=\int_{\partial B(z^*,r')}\frac{P'(x)}{P(x)}dx.
\]
However, it follows from another application of the argument principle, and the fact that $P(x)$ has a single zero in  $\overline{B(z^*,r')}$ at $z^*,$ that the above integral cannot be $0$. This contradiction implies $z^*\in \overline{R(H)}$ and our proof is complete.

\end{proof}

It is natural to wonder whether there exists a set $H$ such that the sets $R(H)$ and $R^{*}(H)$ fill up their ambient annuli, that is $A_{[1/2,2]}$ and $A_{[1/2,1)}$ respectively. In fact such a $H$ cannot exist. For any $H\subseteq S^1$ there exists $z\in\mathbb{C}$ with modulus $1/2$ and $\delta>0,$ such that $R(H)\cap B(z,\delta)=\emptyset$ and $R^{*}(H)\cap B(z,\delta)=\emptyset$. This is because of the following simple reasoning. Since $H$ is a finite set there exists $z\in\mathbb{C}$ such that $|z|=1/2$ and $$|a_i+a_jz|>1/2$$ for all $a_i,a_j\in H$. Equivalently
\begin{equation}
\label{strict}
|a_i+a_jz|>\frac{|z|^2}{1-|z|}
\end{equation} for all $a_i,a_j\in H$. By continuity equation \eqref{strict} holds under small pertubations of $z$. Therefore there must exist $\delta>0,$ such that for all $z'\in B(z,\delta)$ we have $$|a_i+a_jz'|>\frac{|z'|^2}{1-|z'|}.$$ Since $$|\sum_{n=2}^{k}a_n(z')^n|\leq \frac{|z'|^2}{1-|z'|}$$ for all $(a_n)\in H^{k}$ and $k\in \mathbb{N},$ it follows that $z'$ cannot be the zero of a power series or a polynomial. Therefore we must have $R(H)\cap B(z,\delta)=\emptyset$ and $R^{*}(H)\cap B(z,\delta)=\emptyset$.

 \section{Proof of Theorem \ref{main theorem}}
We now turn our attention to proving Theorem \ref{main theorem}. We start with the following technical proposition.

\begin{prop}
\label{beta}
Let $z\in A_{(1/2,1)}$. Suppose $H$ is $2\cos^{-1}(\frac{5-4|z|^2}{4})$-dense, then for any $z'\in \overline{B(0,2)}$ there exists $a\in H$ such that $z^{-1}(z'-a)\in \overline{B(0,2)}$.

\end{prop}

\begin{proof}
Let us start by fixing $z'\in \overline{B(0,2)}.$ Consider the point $z'z^{-1}.$ Clearly $z'z^{-1}\in\overline{B(0,2|z|^{-1})}.$ Let $$S_{z'z^{-1}}^{|z|^{-1}}:=\{\omega\in \mathbb{C}:|\omega-z'z^{-1}|=|z|^{-1}\}.$$ Since $z\in A_{(1/2,1)}$ we must have $S_{z'z^{-1}}^{|z|^{-1}}\cap \overline{B(0,2)}\neq \emptyset$. In fact this intersection must contain an arc of $S_{z'z^{-1}}^{|z|^{-1}}$. This arc is parameterised by two radii of $S_{z'z^{-1}}^{|z|^{-1}}$ with interior angle $\theta$. See Figure \ref{Fig1} for a diagram describing the intersection of  $S_{z'z^{-1}}^{|z|^{-1}}$ with  $\overline{B(0,2)}.$ It is easy to see that the angle $\theta$ is minimised when $z'z^{-1}$ is as far from the origin as possible, i.e., when $z'$ has modulus $2$. Employing elementary techniques from geometry we can see that the angle $\theta$ is at least twice the size of a particular angle of the triangle whose sides have length $|z|^{-1},$ $2,$ and $2|z|^{-1}$ (see Figure \ref{Fig1}). Therefore we can use the well known cosine rule from trigonometry to show that $\theta$ is always bounded below by $$2\cos^{-1}\Big(\frac{5-4|z|^2}{4}\Big).$$

\begin{figure}[h]
\begin{center}
\begin{tikzpicture}
\draw[thick,->] (-4.5,0) -- (4.5,0) ;
\draw[thick,->] (0,-4.5) -- (0,4.5) ;
\draw (0,0) circle (2.5cm);
\draw(-2.5,2.5) circle(2cm);
\node[anchor=south east] at (-2.4,2.5) {$z'z^{-1}$};
\fill (-2.5,2.5) circle[radius=2pt];
\draw (-2.5,2.5) -- (-0.50,2.46);
\draw (-2.5,2.5) -- (-2.46,0.5);
\node[anchor=north west] at (-2.5,2.5) {$\theta$};
\draw [dashed] (0,0) -- (-0.50,2.46);
\draw [dashed] (0,0) -- (-2.50,2.5);
\draw [dashed] (0,0) --  (-2.46,0.5);
\end{tikzpicture}
\end{center}
\caption{A diagram of $S_{z'z^{-1}}^{|z|^{-1}}$ intersecting  $\overline{B(0,2)}$.}
\label{Fig1}
\end{figure}
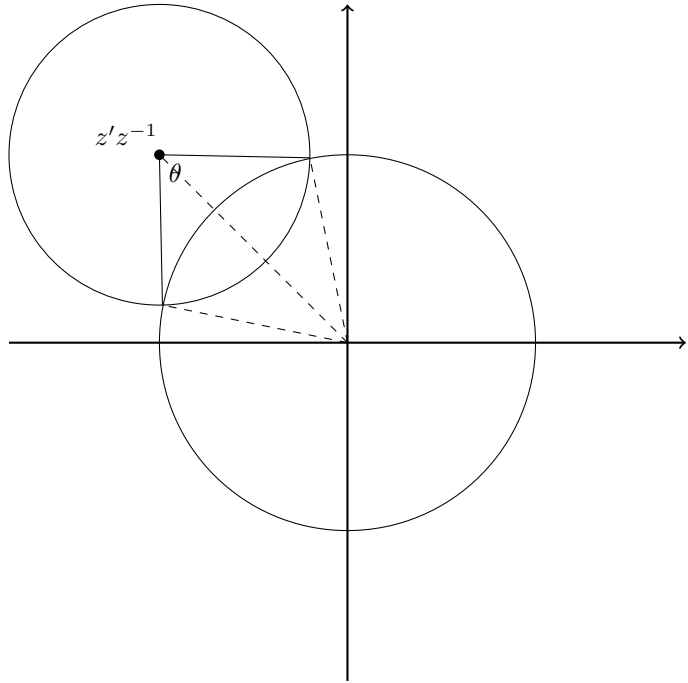

Since $H$ is $2\cos^{-1}\Big(\frac{5-4|z|^2}{4}\Big)$-dense as a subset of $S^1,$ there must exist $a\in H$ such that $z'z^{-1}-az^{-1}$ is contained in the arc of $S_{z'z^{-1}}^{|z|^{-1}}$ which intersects $\overline{B(0,2)}.$ In particular, for this choice of $a$ we have $z^{-1}(z'-a)\in \overline{B(0,2)}.$ 
\end{proof}

Theorem \ref{main theorem} now follows almost immediately from Proposition \ref{beta}.

\begin{proof}[Proof of Theorem \ref{main theorem}]
By the relations given in Proposition \ref{root relations} to prove Theorem \ref{main theorem} it is sufficient to only prove the statement relating to $R^{*}(H)$. Fix $r\in(1/2,1)$ and let $H$ be a $2\cos^{-1}\Big(\frac{5-4r^2}{4}\Big)$-dense subset of $S^1$. Note that $H$ is automatically $2\cos^{-1}\Big(\frac{5-4|z|^2}{4}\Big)$-dense for any $z\in A_{[r,1)}$. So we can apply Proposition \ref{beta} for any $z\in A_{[r,1)}$. Let us now fix $z\in A_{[r,1)}$ and apply Proposition \ref{beta} when $z'=0.$ So there exists $a_0\in H$ such that $x_0:=z^{-1}(-a_0)\in \overline{B(0,2)}.$ Rearranging yields $$0=a_0+x_0z.$$ Applying Proposition \ref{beta} again with $x_0$ in the place of $z'$ yields $a_1$ and $x_1:=z^{-1}(x_0-a_1),$ such that $x_1\in \overline{B(0,2)}$ and
\begin{equation}
\label{substitution}
0=a_0+a_1z+x_1z^2.
\end{equation} One can then apply Proposition \ref{beta} with $z'=x_1$ . Repeating this procedure indefinitely yields a sequence $(a_n)$ and $(x_n)$ such that $x_{n+1}=z^{-1}(x_n-a_{n+1})$ for all $n\in\mathbb{N}$. The terms in $(x_n)$ remain in $\overline{B(0,2)}.$ Therefore we are able to repeatedly apply the substitution $x_{n+1}=z^{-1}(x_n-a_{n+1})$ in \eqref{substitution} and we obtain $$0=\sum_{n=0}^{\infty}a_nz^n.$$ Therefore $z\in R^{*}(H)$. Since $z$ was arbitrary we have $A_{[r,1)}\subseteq R^{*}(H)$.

\end{proof}
The proof of Theorem \ref{main theorem} was based upon ideas from $\beta$-expansions. The argument given relied upon adapting methods from \cite{BakG} and \cite{Parry}.  The proof can easily be adapted to show that under the hypothesis of the theorem, for every $z'\in \overline{B(0,2)}$ there exists $(a_n)\in H^{\mathbb{N}}$ such that $\sum_{n=0}^{\infty}a_nz^n=z'$.
\\

\section{Some further problems}
There are some more challenging problems related to root sets $R(H),R^*(H).$ We mentioned in the beginning of this paper that the exist various results of multiple roots (\cite{BBBP},\cite{BEL}). We say that a $z\in\mathbb{C}$ is a multiple root of a holomorphic function $f$ of order $k$ if for all integers $i=0,1,2,...,k$
\[
f^{(i)}(z)=0.
\]   
Adopting the notation in this paper we can define for any integer $k\geq 0$:
$$R_k(H):=\{z\in \mathbb{C}: \exists k>0,\, \exists (a_n)\in H^{k+1}, P(w)=\sum_{n=0}^{k}a_{n}w^n, \text{ $z$ is a $k$-th order root of } P(w)\}.$$
$$R_k^*(H):=\{z\in \mathbb{C}: \exists (a_n)\in H^{\mathbb{N}}, P(w)=\sum_{n=0}^{\infty}a_{n}w^n, \text{ $z$ is a $k$-th order root of } P(w)\}.$$ 
Not so much has been studied about the above multiple root set, some partial results can be found in \cite{SS}. We can for example consider the following questions:
\begin{itemize}
	\item Are $R_k(H),R_k^*(H)$ dense in any non-degenerate annulus?
	\item What about the connectness and path-connectness of $R_k(H),R_k^*(H)$? 	
	\item What can we see about the boundary of $R_k(H),R_k^*(H)$? 
\end{itemize}

\noindent \textbf{Acknowledgements.} The first author is supported by the EPSRC grant EP/M001903/1. The second author is supported by a PhD scholarship provided by the School of
Mathematics in the University of St Andrews. The authors are grateful to Jonathan Fraser, Tom Kempton, and Sascha Troscheit for fruitful discussions.

\bibliography{unimodular_polynomials}
\bibliographystyle{amsalpha}

\end{document}